\documentclass[12pt]{amsart}

\usepackage{amsthm,amsfonts,amsmath,amssymb,latexsym,epsfig,mathrsfs,yfonts,marvosym}
\usepackage[usenames]{color}
\usepackage{graphicx}
\usepackage[all]{xy}
\usepackage{epsfig}
\usepackage{epic}
\usepackage{eepic}
\usepackage{hyperref}
\usepackage{dsfont}\let\mathbb\mathds
\usepackage[english]{babel}

\DeclareMathAlphabet\oldmathcal{OMS}        {cmsy}{b}{n}
\SetMathAlphabet    \oldmathcal{normal}{OMS}{cmsy}{m}{n}
\DeclareMathAlphabet\oldmathbcal{OMS}       {cmsy}{b}{n}
 
\usepackage{eucal}

\usepackage[active]{srcltx}

\newtheorem{theorem}{Theorem}[section]
\newtheorem{lemma}[theorem]{Lemma}
\newtheorem{proposition}[theorem]{Proposition}
\newtheorem{corollary}[theorem]{Corollary}

\newtheorem{def/prop}[theorem]{Definition/Proposition}

\newenvironment{example}{\medskip \refstepcounter{theorem}
\noindent  {\bf Example \thetheorem}.\rm}{\,}
\newenvironment{remark}{\medskip \refstepcounter{theorem}
\newcommand     {\comment}[1]   {}
\newcommand{\mute}[2] {}
\newcommand     {\printname}[1] {}

\noindent  {\bf Remark \thetheorem}.\rm}{\,}
\newtheorem*{ack}{Acknowledgements}

\def\<{\langle}

\def\>{\rangle}

\def\BOne{{\mathchoice {\rm 1\mskip-4mu l} {\rm 1\mskip-4mu l}
                          {\rm 1\mskip-4.5mu l} {\rm 1\mskip-5mu l}}}

\def\fract#1#2{\raise4pt\hbox{$ #1 \atop #2 $}}
\def\decdnar#1{\phantom{\hbox{$\scriptstyle{#1}$}}
\left\downarrow\vbox{\vskip15pt\hbox{$\scriptstyle{#1}$}}\right.}

\def\bbc{{\mathbb C}}

\def\bbp{{\mathbb P}}
\def\bbq{{\mathbb Q}}
\def\bbr{{\mathbb R}}

\def\bbz{{\mathbb Z}}

\def\grb{\beta}

\def\grd{\delta}

\def\grg{\gamma}
\def\gri{\iota}
\def\grk{\kappa}

\def\gro{\omega}

\def\grr{\rho}

\def\grD{\Delta}

\def\grS{\Sigma}

\def\bfl{{\bf l}}

\def\bfv{{\bf v}}
\def\bfw{{\bf w}}

\def\cald{{\mathcal D}}

\def\calf{{\mathcal F}}

\def\calh{{\mathcal H}}
\def\cali{{\mathcal I}}

\def\cals{{\oldmathcal S}}

\def\calw{{\mathcal W}}

\def\la#1{\hbox to #1pc{\leftarrowfill}}
\def\ra#1{\hbox to #1pc{\rightarrowfill}}

\def\gp{{\mathfrak p}}

\def\gr{{\mathfrak r}}

\def\gt{{\mathfrak t}}

\def\gz{{\mathfrak z}}

\def\gC{{\mathfrak C}}

\def\hook{\mathbin{\hbox to 6pt{%
                 \vrule height0.4pt width5pt depth0pt
                 \kern-.4pt
                 \vrule height6pt width0.4pt depth0pt\hss}}}

\def\12{\xi_{k_1,k_2}}
\def\m5{M^5_{k_1,k_2}}

\begin{document}

\title{On Positivity in Sasaki Geometry}

\author{Charles P. Boyer and Christina W. T{\o}nnesen-Friedman}\thanks{Both authors were partially supported by grants from the Simons Foundation, CPB by \#245002 and \#519432 and CWT-F by \#208799 and \#422410.} 
\address{Charles P. Boyer, Department of Mathematics and Statistics,
University of New Mexico, Albuquerque, NM 87131.}
\email{cboyer@unm.edu} 
\address{Christina W. T{\o}nnesen-Friedman, Department of Mathematics, Union
College, Schenectady, New York 12308, USA } \email{tonnesec@union.edu}

\date{\today}


\maketitle

\markboth{Sasaki Join, Positive Ricci Curvature}{Charles P. Boyer and Christina W. T{\o}nnesen-Friedman}

\begin{abstract}
It is well known that if the dimension of the Sasaki cone $\gt^+$ is greater than one, then all Sasakian structures in $\gt^+$ are either positive or indefinite. We discuss the phenomenon of type changing within a fixed Sasaki cone. Assuming henceforth that $\dim\gt^+>1$ there are three possibilities, either all elements of $\gt^+$ are positive, all are indefinite, or both positive and indefinite Sasakian structures occur in $\gt^+$. We illustrate by examples how the type can change as we move in $\gt^+$.  If there exists a Sasakian structure in $\gt^+$ whose total transverse scalar curvature is non-positive, then all elements of $\gt^+$ are indefinite. Furthermore, we prove that if the first Chern class is a torsion class or represented by a positive definite $(1,1)$ form, then all elements of $\gt^+$ are positive.
\end{abstract}

\section{Introduction}
The main purpose of this note is to understand how Sasakian structures change as one moves through the Sasaki cone. When the dimension of the Sasaki cone $\gt^+$ is greater than $1$, the type of a Sasakian structure in $\gt^+$ must be either positive or indefinite, that is, the basic first Chern class $c_1(\calf_\xi)$ can be represented either by a positive definite or an indefinite $(1,1)$-form with respect to its transverse holomorphic CR structure $(\cald,J)$. We show by example that the type can change as we move in the Sasaki cone. Thus, the type is an invariant of the Sasakian structure, but not of the underlying CR structure.

Furthermore, it follows from the transverse Yau Theorem of El Kacimi-Alaoui \cite{ElK} that every positive Sasakian structure can be deformed by a transverse homothety to a Sasakian structure with positive Ricci curvature (cf. Theorem 7.5.20 of \cite{BG05}). This procedure was exploited in \cite{BGN03a,BGN03b,BG05h} to obtain results  concerning Sasakian structures with positive Ricci curvature. We should mention, however, that Propositions 2.6, 2.8 and Theorem 2.9 in \cite{BGN03a} concerning the necessity of being a spin manifold are incorrect. Non-spin Sasaki manifolds can have Sasakian metrics of positive Ricci curvature (cf Chapter 10 of \cite{BG05}). The point is that the positive definite transverse $(1,1)$-form representing $c_1(\calf_\xi)$ need not necessarily be the transverse K\"ahler metric of a Sasakian structure on $M$. In Appendix A we give some new explicit examples of Sasaki manifolds with positive Ricci curvature, including non-spin examples. 

\begin{ack}
The authors are grateful to Vestislav Apostolov for carefully reading our paper and suggesting some important clarifications.
We also thank Hongnian Huang, and Eveline Legendre for their interest in and comments on our work.
\end{ack}

\section{The Sasaki Cone}\label{Sascone}
In what follows our Sasaki manifolds $M$ are both oriented and co-oriented, that is we fix both the orientation of $M$ and the orientation of contact bundle $\cald$. 
A CR structure $(\cald,J)$ on a manifold $M$ is said to be of Sasaki type if there is a Sasakian structure whose underlying CR structure is $(\cald,J)$. CR structures of Sasaki type are strictly pseudoconvex. The space of Sasakian structures belonging to a CR structure $(\cald,J)$ of Sasaki type has a subspace which is identified with an open cone in the Lie algebra $\gt_k(\cald,J)$ of a maximal torus $T$ in the CR automorphism group $\gC\gr(\cald,J)$ where $k$ is the dimension of $T$. This subspace is called the {\it unreduced Sasaki cone} and is defined by 
\begin{equation}\label{unsascone}
\gt^+_k(\cald,J)=\{\xi'\in\gt_k(\cald,J)~|~\eta(\xi')>0\}
\end{equation}
where $(\xi,\eta,\Phi,g)$ is a fixed Sasakian structure of $(\cald,J)$.
It is easy to see that $\gt^+_k(\cald,J)$ is a convex cone in $\gt_k=\bbr^k$.
Then the {\it reduced Sasaki cone} is defined by $\grk(\cald,J)=\gt^+_k(\cald,J)/\calw(\cald,J)$ where $\calw(\cald,J)$ is the Weyl group of $\gC\gr(\cald,J)$. The reduced Sasaki cone $\grk(\cald,J)$ can be thought of as the moduli space of Sasakian structures whose underlying CR structure is $(\cald,J)$. We shall often suppress the CR notation $(\cald,J)$ when it is understood from the context. We also refer to a Sasakian structure $\cals$ as an element of $\gt^+_k(\cald,J)$. We view the Sasaki cone $\gt^+_k(\cald,J)$ as a $k$-dimensional smooth family of Sasakian structures.

We recall the {\it type} of a Sasakian structure \cite{BGN03a,BG05}.  A Sasakian structure $\cals=(\xi,\eta,\Phi,g)$ is {\it positive (negative)} if the basic first Chern class $c_1(\calf_\xi)$ is represented by a positive (negative) definite basic $(1,1)$-form. It is {\it null} if $c_1(\calf_\xi)=0$, and {\it indefinite} if $c_1(\calf_\xi)$ is otherwise. It is well known that for negative and null Sasakian structures the connected component of the Sasaki automorphism group is the circle group $S^1$ generated by the Reeb vector field, thus, $k=1$. So when $k=\dim \gt^+>1$ the type of any element in $\gt^+$ is either positive or indefinite. The type is an invariant of the transverse homothety class of $\xi$, that is $a\xi$ and $\xi$ have the same type for all $a\in\bbr^+$. We denote by $\gp^+\subset \gt^+$ the subset of positive Sasakian structures, and call it the {\it positive Sasaki subset}. It is convenient to introduce the notation $\grb>0$ for elements $\grb$ of the Abelian groups $H^{1,1}_B(\calf_\xi),H^2(M,\bbr),H^2(M,\bbz)$ to mean that $\grb$ can be represented by a positive definite $(1,1)$-form. 

By a conical subset we mean that if $\xi\in\gp^+$ then so is $a\xi$ for any $a\in\bbr^+$.

\begin{proposition}\label{posopen}
The positive Sasaki subset $\gp^+$ is an open conical subset of $\gt^+$.
\end{proposition}

\begin{proof}
Positivity is an open condition, and $\gp^+$ is conical since the transverse Ricci form $\grr^T_\xi$ is invariant under scaling of $\xi$.  
\end{proof}

So if $\gt^+$ contains a positive Sasakian structure it contains an open set of positive Sasakian structures; hence, $\gp^+$ contains a quasi-regular Sasakian structure. Our next result relates the existence of positive Sasakian structures in $\gt^+$ with the total transverse scalar curvature defined by 
\begin{equation}\label{S_xidef}
{\bf S}_\xi=\int_M s^T_\xi dv_\xi
\end{equation}
where $\xi\in\gt^+$. We mention that, as with the volume, ${\bf S}_\xi$ depends only on the isotopy class of the Sasakian structure. Our results make use of the remarkable Equation (32) from \cite{BHL17} which stated as Lemma 5.2 and Proposition 5.3 there, becomes

\begin{theorem}[\cite{BHL17}]\label{BHLthm}
The following hold:
\begin{enumerate}
\item If there exists a Sasakian structure $\xi_0\in\gt^+$ whose transverse scalar curvature is positive almost everywhere, then ${\bf S}_\xi>0$ for all $\xi\in\gt^+$;
\item If $\gp^+$ is non-empty, the total transverse scalar curvature ${\bf S}_\xi$ is positive for all $\xi\in\gt^+$. Alternatively stated, if $\dim\gt^+>1$ and there exists $\xi\in\gt^+$ whose total transverse scalar curvature ${\bf S}_\xi$ is non-positive, then all Sasakian structures in $\gt^+$ are indefinite.
\end{enumerate}
\end{theorem}

It is well known that the positivity of a Sasakian structure depends only on the homothety class, that is it only depends on the ray in $\gt^+$. However, both the volume ${\bf V}_\xi$ and the total transverse scalar curvature ${\bf S}_\xi$ vary with the point on the ray. It is, thus, convenient to consider the Einstein-Hilbert functional ${\bf H}(\xi)$ \cite{BHLT15,BHL17}.
Actually, it is more convenient to consider the `signed' version
\begin{equation}\label{signEH}
{\bf H}_1(\xi)= {\rm sign}({\bf S}_\xi)\frac{|{\bf S}_\xi|^{n+1}}{{\bf V}^n_\xi},
\end{equation}
when the dimension of the Sasaki manifold is $2n + 1$. 
It was shown in \cite{BHL17} that ${\bf H}_1(\xi)$ tends to $+\infty$ as $\xi$ approaches the boundary, thus, ${\bf H}_1(\xi)$ has a global minimum $\xi_{min}$. The critical points of ${\bf H}_1(\xi)$ in $\gt^+$ are ${\bf S}_\xi=0$ and the zeros of the Sasaki-Futaki invariant \cite{BHLT15}. In terms of the Einstein-Hilbert functional, item (2) of Theorem \ref{BHLthm} takes the form
\begin{theorem}[\cite{BHL17}]\label{scalarpos2}
If ${\bf H}_1(\xi_{min})\leq 0$, all Sasakian structures in $\gt^+$ are indefinite. Alternatively, if $\gp^+$ is non-empty, ${\bf H}_1(\xi_{min})>0$.
\end{theorem}

The hypothesis here or that of (2) of Theorem \ref{BHLthm} is significantly weaker than that of (1) of Theorem \ref{BHLthm}. Thus, one may wonder whether there exists contact structures of Sasaki type whose entire Sasaki cone consists of indefinite Sasakian structures, but ${\bf H}_1(\xi_{min})>0$. The next example indicates that this is fairly common.

\begin{example}\label{Riesurfex} Let $N=\grS_g$ be a Riemann surface of genus $g\geq 2$ and consider Examples 5.7 and  5.15 of \cite{BoTo13} which describe the $S^3_\bfw$ join $M_{g-1,1,12,1}$ as an $S^3$ bundle over $\grS_g$ (see the next section for a brief description of this join construction). It is the trivial $S^3$ bundle if $g$ is odd, and the non-trivial bundle if $g$ is even. Since in this case $N$ is not Fano and $\gt^+_\bfw=\gt^+$, Corollary \ref{openpos} below implies that the entire Sasaki cone consists of indefinite Sasakian structures. 
By Equation (51) of \cite{BoTo13} these Sasaki manifolds have a unique CSC ray which by Example 5.15 has positive constant transverse scalar curvature\footnote{The transverse scalar curvature in Example 5.15 of \cite{BoTo13} should be $8\pi$ not $8\pi /3$.} $8\pi$ independent of $g$.
By Theorem \ref{BHLthm} we then know that ${\bf S}_\xi>0$ for all $\xi\in\gt^+$ and so ${\bf H}_1(\xi_{min})>0$, and, by Theorem 1.4 of \cite{BHLT15} together with Proposition 5.10 of \cite{BoTo13}, the unique CSC ray is the unique
critical point, $\xi_{min}$, of ${\bf H}_1$ (since we are on a five dimensional manifold, ${\bf H}_1$ is just the Hilbert-Einstein functional here). Summarizing we have\begin{proposition}\label{indefH1pos}
For each $g\geq 2$ the contact manifold $M_{g-1,1,12,1}$ of Sasaki type has a Sasaki cone saturated by indefinite Sasakian structures and with ${\bf H}_1(\xi_{min})>0$.
\end{proposition}

\end{example}

We will see later by examples that although ${\bf S}_{\xi}>0$ for all elements $\xi\in\gt^+$, there can be a range of positivity that is not all of $\gt^+$, that is $\gp^+$ can be a proper subset of $\gt^+$. Outside of the range of positivity the Sasakian structures are indefinite.

\begin{remark}\label{defrem}
A recent result of Nozawa  \cite{Noz14} shows that deformations of the transverse holomorphic foliation of a Sasakian structure do not generally remain Sasakian. The obstruction lies in the basic Hodge cohomology group $H^{0,2}_B(\calf_\xi)$. However, for positive Sasakian structures, there is a vanishing theorem $H^{0,q}_B(\calf_\xi)=0$ \cite{BGN03a}, so such small deformations of positive Sasakian structures remain positive which is consistent with Corollary \ref{openpos} below.
\end{remark}

In the special case $c_1(\cald)=0$ we have
\begin{theorem}\label{c10pos}
Let $(M,\cals)$ be a Sasaki manifold with $c_1(\cald)_\bbr=c_1(\cald)\otimes\bbr=0$ and suppose that its Sasaki cone $\gt^+$ contains a positive Sasakian structure. Then every element of the Sasaki cone is positive, that is $\gt^+=\gp^+$.
\end{theorem}

\begin{proof}
From the exact sequence in Lemma 7.5.22 of \cite{BG05} we have $c_1(\calf_\xi)=a[d\eta]_B$ for some $a\in\bbr$. If $a\leq 0$ for some Sasakian structure $\cals\in\gt^+$ then $\dim\gt^+=1$ and $c_1(\calf_\xi)\leq 0$ which contradicts the hypothesis that $\gt^+$ contains a positive element. So we must have $a>0$ with $c_1(\calf_\xi)=a[d\eta]_B$ for any $\xi\in\gt^+$. This implies that the entire cone is positive.
\end{proof}

We want to understand generally the effect of $c_1(\cald)$ on the Sasakian structures in $\gt^+$. For the positive case we have

\begin{theorem}\label{pos11form}
Let $(\cald,J)$ be a CR manifold of Sasaki type.
If $c_1(\cald)>0$ then all Sasakian structures in $\gt^+$ are positive, hence, $\gt^+=\gp^+$.
\end{theorem}

\begin{proof}
From Lemma 7.5.22 of \cite{BG05} for any Sasakian structure $\cals=(\xi,\eta,\Phi,g)$ in $\gt^+$ we have the split exact sequence
$$0\ra{2.5} \bbr\fract{\grd}{\ra{2.5}} H^2_B(\calf_\xi)\fract{\gri_*}{\ra{2.5}} H^2(M,\bbr)\ra{2.5}H^1_B(\calf_\xi),  $$
with $\grd(a)=a[d\eta]_B$ and $\gri_*c_1(\calf_\xi)=c_1(\cald)$. Now $c_1(\cald)$ is represented by a positive $(1,1)$-form, and from the definition of $\gri_*$ it is the same $(1,1)$ form that represents $c_1(\calf_\xi)$. Thus, by El Kacimi Alaoui's theorem $c_1(\cald)$ can be represented by $1/2\pi$ times the transverse 
Ricci form $\grr^T_\xi$ of a Sasakian structure $\cals$ with underlying CR structure $(\cald,J)$, cf. Theorem 7.5.20 of \cite{BG05}. Moreover, $\grr^T_\xi$ is positive.  Since $\cals$ is an arbitrary element of $\gt^+$, all Sasakian structures in $\gt^+$ are positive.
\end{proof}

We now consider the general case. Assume that $\dim \gt^+>1$ and for simplicity that $H^1(M,\bbr)=0$. Then from Proposition 7.2.3 of \cite{BG05} $H^1_B(\calf_\xi)\approx H^1(M,\bbr)=0$, so there exists a monomorphism $s_*:H^2(M,\bbr)\ra{1.8} H^2_B(\calf_\xi)$ and an $a_\xi\in\bbr$ such that 
\begin{equation}\label{c1poseqn}
c_1(\calf_\xi)=s_*c_1(\cald)+a_\xi[d\eta]_B.
\end{equation}
If $c_1(\cald)=0$ we recover Theorem \ref{c10pos} and when $c_1(\cald)>0$ Thereom \ref{pos11form} forces $c_1(\calf_\xi)$ to be positive. For any other case $c_1(\calf_\xi)$ can be either positive or indefinite depending on choices. This gives rise to {\it type changing} which we discuss for the $\bfw$ subcone of an $S^3_\bfw$-join in Section \ref{posinsascone}.


Let us assume that $\gp^+\neq \emptyset$.  Given a non-zero first Chern class $c_1(\cald)\in H^2(M,\bbz)$ assume that there exists a non-negative real number $B'$  and a primitive class $\grg\in H^2(M,\bbz)$ such that $c_1(\cald)+B'\grg>0$. We define $B$ by
\begin{equation}\label{supB}
B=\inf\{B'\in\bbr_{\geq 0} ~|~ c_1(\cald)+B'\grg>0\}.
\end{equation} 
$B$ is an invariant of the family of Sasakian structures defined by $\gt^+$. We give an explicit value of $B>0$ for certain contact structures of Sasaki type on $S^2\times S^3$ described in Proposition \ref{typech} below and show that it is related to the map $g$ in \cite{BoPa10}.

\section{The Weighted $S^3$ Join Construction}\label{join}
We give a brief review of the weighted $S^3_\bfw$ join construction. We refer to our previous papers \cite{BoTo13,BoTo14NY,BoTo14a} for more details. Our construction is given a regular Sasaki manifold $M$ as the total space of an $S^1$-bundle over the Hodge manifold $N$ with primitive K\"ahler class $[\gro_N]$, and a pair of relatively prime positive integers $\bfl=(l_1,l_2)$, we define the $\bfl=(l_1,l_2)$-join of $M$ with $S^3_\bfw$ to be the quotient of $M\times S^3$ by the $S^1$-action 
\begin{equation}\label{joindefeq}
(x,u;z_1,z_2)\mapsto (x,e^{il_2\theta}u;e^{-il_1w_1\theta}z_1,e^{-il_1w_2\theta}z_2)
\end{equation}
where $u$ denotes the fiber of the natural projection $\pi:M\ra{1.6} N$, and $|z_1|^2+|z_2|^2=1$ describes the unit sphere in $\bbc^2$. For this quotient to be smooth we must impose $\gcd(l_2,l_1w_1w_2)=1$ and we denote the quotient by $M_{\bfl,\bfw}=M\star_\bfl S^3_\bfw$. Now $M_{\bfl,\bfw}$ has an induced Sasakian structure whose contact form $\eta_{\bfl,\bfw}$ and Reeb vector field $\xi_{\bfl,\bfw}$ satisfy 
\begin{equation}\label{sasjoinxieta}
\pi_L^*\eta_{\bfl,\bfw}=l_1\eta_M+l_2\eta_\bfw, \qquad \xi_{\bfl,\bfw}=(\pi_L)_*\bigl(\frac{1}{2l_1}\xi_M+\frac{1}{2l_2}\xi_\bfw\bigr),
\end{equation}
where $\eta_M,\xi_M,(\eta_\bfw,\xi_\bfw)$ are the contact 1-form and Reeb vector field on $M,(S^3)$, respectively. Since the kernel of $(\pi_L)_*$ is generated by the vector field $L_{\bfl,\bfw}=\frac{1}{2l_1}\xi_M-\frac{1}{2l_2}\xi_\bfw$, the Reeb vector field of $\eta_{\bfl,\bfw}$ on $M_{\bfl,\bfw}$ is $\frac{1}{l_2}\xi_\bfw$. Furthermore, the quotient by the $S^1$-action of $M_{\bfl,\bfw}$ generated by $\xi_{\bfl,\bfw}$ is the projective algebraic orbifold $N\times \bbc\bbp^1[\bfw]$ with K\"ahler form $\gro_{\bfl,\bfw}$ satisfying $\pi^*\gro_{\bfl,\bfw}=d\eta_{\bfl,\bfw}$. So the join fits into the commutative diagram
\begin{equation}\label{s2comdia}
\begin{matrix}  M\times S^3_\bfw &&& \\
                          &\searrow\pi_L && \\
                          \decdnar{\pi_{2}} && M_{\bfl,\bfw} &\\
                          &\swarrow\pi_1 && \\
                         N\times\bbc\bbp^1[\bfw] &&& 
\end{matrix}
\end{equation}  
where the $\pi$s are the obvious projections. Without loss of generality we can take $w_1\geq w_2$. Since $\bbc\bbp^1[\bfw]$ has a Hamiltonian vector field which lifts to $M_{\bfl,\bfw}$, the latter has a Sasaki cone of dimension at least two. This two dimensional cone is generated by the lifted Hamiltonian vector field and the Reeb vector field $\xi_{\bfl,\bfw}$ and can be identified with the open first quadrant of $\bbr^2$, described by $\{(v_1,v_2)~|~v_1,v_2>0\}$. It is called the $\bfw$-Sasaki cone and it denoted by $\gt^+_{\bfl,\bfw}$. We shall abuse notation somewhat saying that $\bfv$ or a Sasakian structure belongs to $\gt^+_{\bfl,\bfw}$. This $\bfw$-Sasaki cone is associated to the underlying strictly pseudoconvex CR structure $(\cald_{\bfl,\bfw},J)$. 

Note that by taking $\pi^*[\gro_N]= l_2\grg$ and $\pi^*[\gro_\bfw]=-l_1\grg$ for some generator $\grg\in H^2(M_{l_1,l_2,\bfw},\bbz)$, where $\pi^*\gro_N =d \eta_M$,
$\pi^*\gro_\bfw = d\eta_\bfw$, and $\pi$ is the obvious projection in each case, we have
\begin{equation}\label{c1cald}
c_1(\cald_{l_1,l_2,\bfw})=\pi^*c_1(N)-l_1|\bfw|\grg.
\end{equation}

If we choose a quasi-regular Reeb field in $\gt^+_{\bfl,\bfw}$, its quotient orbifold is a ruled manifold of the form $S_n=\bbp(\BOne\oplus L_n)$ where $L_n$ is a complex line bundle of `degree' $n$ (see Section 2 of \cite{BoTo14a} for the meaning of `degree'), and with an orbifold structure given by a branch divisor of the form
\begin{equation}\label{orbbranch}
\grD= (1-\frac{1}{m_1})D_1+ (1-\frac{1}{m_2})D_2,
\end{equation}
where $D_1 (D_2)$ are the zero (infinity) section of $L_n$, respectively, and $m_i$ is the ramification index of the branch divisor $D_i$. The pair $(S_n,\grD)$ is called a {\it log pair}.
The values $n$ and $m_i$ are given by
\begin{equation}\label{ramind}
n =\frac{l_1}{s}(w_1v_2-w_2v_1), \qquad m_i =mv_i 
\end{equation} 
for $i=1,2$, where $s=\gcd(|w_2v_1-w_1v_2|,l_2)$, and $m = \frac{l_2}{s}$.

\section{Positivity in the $\bfw$-Sasaki Cone}\label{posinsascone}
Here we study the so-called $\bfw$-Sasaki cone for the 4-parameter class of Sasaki manifolds $M_{\bfl,\bfw}$. In particular, we are interested in how the type of a Sasaki manifold changes as we move in the Sasaki cone. We denote by $\gp^+_\bfw$ the intersection of $\gp^+$ with the subcone $\gt^+_\bfw$.

\subsection{$c_1^{orb}$ and $c_1(\calf_\xi)$}
Choose a basis $\{\grb_i\}_{i=0}^k$ of $H^2(N,\bbz)/{\rm Tor}\approx \bbz^{k+1}$ that diagonalizes $c_1(N)$ and such that $\grb_0=[\gro_N]$.
We have 
\begin{equation}\label{c1diag}
c_1(N)=b_0[\gro_N]+\sum_{i=1}^kb_i\grb_i,
\end{equation}
and
\begin{equation}\label{c1orbdiag}
c_1^{orb}(N\times \bbc\bbp^1[\bfw])=c_1(N)+c_1^{orb}(\bbc\bbp^1[\bfw])= b_0[\gro_N]+\sum_{i=1}^kb_i\grb_i+ |\bfw|\grb_{k+1}
\end{equation}
where $\grb_{k+1}=\frac{[\gro_0]}{w_1w_2}$ is an integral class in the orbifold cohomology group $H^2_{orb}(\bbc\bbp^1[\bfw],\bbz)$ although it is a rational class in the ordinary cohomology $H^2(\bbc\bbp^1,\bbq)$. Note that $c_1^{orb}(N\times \bbc\bbp^1[\bfw])$ is positive if and only if $c_1(N)$ is positive. Now $c_1^{orb}$ pulls back to $c_1(\calf_{\xi_\bfw})$ on $M_{\bfl,\bfw}$, so we have
\begin{equation}\label{c_1calf}
c_1(\calf_{\xi_\bfw})=\pi^*c_1^{orb}(N\times \bbc\bbp^1[\bfw])=\sum_{i=0}^{k+1}b_i\grg_i,
\end{equation}
where $\grg_i=\pi^*\grb_i$ form a basis of $H^2_B(\calf_\xi)$. Actually the forms $\grb_i$ are $(1,1)$ forms on $N\times \bbc\bbp^1[\bfw]$, so the forms $\grg_i$ are basic $(1,1)$-forms, i.e. they are in $H^{1,1}_B(\calf_\xi)$. Furthermore, $c_1(\calf_{\xi_\bfw})$ is an integral class.

Now consider a general quasi-regular Sasakian structure $\cals_\bfv$ with Reeb vector field $\xi_\bfv$ in the $\bfw$-Sasaki cone $\gt^+_\bfw$. The base orbifold $(S_n,\grD)$ has orbifold first Chern class
\begin{equation}\label{Snorbc1}
c_1^{orb}(S_n,\Delta) = c_1(N) + \frac{1}{m_1}PD(D_1)+\frac{1}{m_2}PD(D_2),
\end{equation}
where $c_1(N)$ is viewed as a pull-back to $S_n$ and the ramification indices $m_i$ satisfy Equation \eqref{ramind}. Now we have $\pi_\bfv^*c_1^{orb}(S_n,\Delta)=c_1(\calf_{\xi_\bfv})$, so $c_1^{orb}(S_n,\Delta)$ is positive if and only if $c_1(\calf_{\xi_\bfv})$ is positive. Moreover, it follows from the analysis of Section 5.1 of \cite{BoTo14a} that $c_1(\calf_{\xi_\bfv})$ is smooth function on $\gt^+_\bfw$.

\begin{proposition}\label{c1wc1v}
Consider the $S^3_\bfw$-join $M_{\bfl,\bfw}$. If $c_1(N)$ is positive, then there is a positive Sasakian structure in the $\bfw$-Sasaki cone $\gt_\bfw^+\subset \gt^+$, and ${\bf S}_\xi>0$ for all $\xi\in\gt^+$; hence, ${\bf H}_1(\xi_{min})>0$.
\end{proposition}

\begin{proof}
The first statement follows immediately from Equation \eqref{c_1calf}, and then the second statement follows from Theorem \ref{BHLthm}. 
\end{proof}

The converse to the first statement of Proposition \ref{c1wc1v} is also true, which we prove in Proposition \ref{typech} below. Thus, if a Sasakian structure $\cals_\bfv=(\xi_\bfv,\eta_\bfv,\Phi,g)\in \gt^+_\bfw$ is positive, then $c_1(\calf_{\xi_\bfw})$ is positive. However, generally not all Sasakian structures in $\gt_\bfw^+$ will be positive which leads to type changing in $\gt_\bfw^+$.

\subsection{Type Changing}
We describe this type changing in the case of the $S^3_\bfw$-join with a regular Sasaki manifold $M$. However, before doing so we prove the converse of Proposition \ref{c1wc1v}. This shows that $N$ must be Fano in which case we set $b_0=\cali_N$ which is referred to as the Fano index in the monotone case. Note, however, that we do not restrict to the monotone case. We have

\begin{proposition}\label{typech}
Let $M_{\bfl,\bfw}$ be a smooth $S^3_\bfw$ join as described in Section \ref{join}.
If there is a positive Sasakian structure in the $\bfw$-cone $\gt^+_{\bfl,\bfw}$, then $c_1(N)$ is positive, that is, $N$ is Fano. 
When $N$ is Fano, the following hold:
\begin{enumerate}
\item The entire $\bfw$-cone $\gt^+_{\bfl,\bfw}$ is positive if and only if $l_2\cali_N> l_1w_1$. Equivalently $\gt^+_{\bfl,\bfw}$ is positive everywhere if and only if $c_1(\cald_{\bfl,\bfw})+l_1w_2\grg>0$.  
\item If $\frac{\cali_Nl_2}{w_2l_1}<1$ the range of positivity $\gp^+_\bfw$ is 
$$\frac{w_1}{w_2}-\frac{l_2\cali_N}{l_1w_2}<\frac{v_1}{v_2}<\frac{w_1}{w_2}\frac{1}{1-\frac{l_2\cali_N}{l_1w_2}}.$$
\item If $1\leq \frac{\cali_Nl_2}{w_2l_1}<\frac{w_1}{w_2}$ then $\gp^+_\bfw$ is
$$\frac{w_1}{w_2}-\frac{l_2\cali_N}{l_1w_2}<\frac{v_1}{v_2}.$$
\end{enumerate}
\end{proposition}

\begin{proof}
We prove this first in the quasi-regular case and then by a continuity argument as done in \cite{BoTo14a} the irregular case follows. Given a quasi-regular Reeb vector field $\xi_\bfv$ with $\bfv=(v_1,v_2)$ a pair of relatively prime positive integers, we have a base orbifold of the form $(S_n,\grD)$ where $S_n$ is a smooth projective variety of the form $\bbp(\BOne\oplus L_n)$ over a K\"ahler manifold $N$ with branch divisor $\grD$.  Moreover, from  Equation \eqref{Snorbc1} and the relations
\begin{equation}\label{cohrel}
2\pi(PD(D_1)-PD(D_2))=[\omega_{N_n}], \qquad \omega_{N_n} = 2\pi n\omega_N,
\end{equation}
the orbifold first Chern class is given by
\begin{equation}\label{orbiChern}
c_1^{orb}(S_n,\Delta) = c_1(N) + \bigl(\frac{1}{m_1}-\frac{1}{m_2}\bigr)\frac{[\omega_{N_n}]}{4\pi} +  \bigl(\frac{1}{m_1}+\frac{1}{m_2}\bigr)\frac{(PD(D_1)+PD(D_2))}{2}.
\end{equation}
Now we have an admissible K\"ahler class on $(S_n,\Delta)$ as adapted from \cite{ACGT08} which can be written - up to a positive scaling factor - as
\begin{equation}\label{admKahclass}
\Omega_{\mathbf r}  = [\omega_{N_n}]/r + 2 \pi PD[D_1+D_2],
\end{equation}
where $0<|r|<1$, $\omega_{N_n} = 2\pi n\omega_N$, and $r\cdot n>0$. Note that $r$, $n$ and $\omega_{N_n}$ all have the same sign. Then Equation \eqref{orbiChern} gives
\begin{equation}\label{c1orb2}
c_1^{orb}(S_n,\Delta) =\sum_{i=1}^k b_i\grb_i +\bigl(\frac{2b_0}{n}+\frac{1}{m_1}-\frac{1}{m_2}\bigr)\frac{[\omega_{N_n}]}{4\pi} +  \bigl(\frac{1}{m_1}+\frac{1}{m_2}\bigr)\frac{(PD(D_1)+PD(D_2))}{2}.
\end{equation}
We, therefore, see from the admissible classes \eqref{admKahclass} 
that $c_1^{orb}(S_n,\Delta)$ is positive if and only if $b_i>0$, the classes $\grb_i$ are represented by positive definite $(1,1)$-forms, and the following inequalities hold
 \begin{equation}\label{fano}
 \begin{array}{ccccc}
(\frac{2b_0}{n} + \frac{1}{m_1}- \frac{1}{m_2}) & > & ( \frac{1}{m_1}+ \frac{1}{m_2}) & \text{if} & n>0\\
\\
(\frac{2b_0}{n} +  \frac{1}{m_1}- \frac{1}{m_2}) & < & - ( \frac{1}{m_1}+ \frac{1}{m_2}) & \text{if} & n<0.
\end{array}
\end{equation}
Thus, if $b_0\leq 0$ we get a contradiction since the $m_i$ are positive. This implies the first statement that $N$ must be Fano. So we set $b_0=\cali_N$ the Fano index. Then using Equations \eqref{ramind} we can rewrite these conditions  as 
\begin{equation}\label{poseqns}
0<
\begin{cases}
\cali_Nl_2v_2 -l_1(w_1v_2 - w_2v_1) & \text{if $w_1v_2 - w_2v_1>0$} \\
\cali_Nl_2v_1 +l_1(w_1v_2 - w_2v_1) & \text{if $w_1v_2 - w_2v_1<0$.}
\end{cases}
\end{equation}
Note that sending $n\mapsto -n$ the left hand inequalities become automatic. From this one obtains the result. 
\end{proof}
\begin{remark}
If $N$ is Fano and $(m_1,m_2)=(1,1)$ then the inequalities \eqref{fano} are due to Koiso \cite{Koi90}. See also Theorem 3.1 in \cite{ACGT08b}.
\end{remark}
\begin{remark}
Notice that in all cases there is some range of positivity; however, the range shrinks as $\frac{l_2\cali_N}{l_1w_2}$ tends to $0$. From  Proposition \ref{typech} we have
\begin{corollary}\label{openpos}
If $N$ is Fano, then there is a connected open subcone of positive Sasakian structures in $\gt^+_\bfw$. If $N$ is not Fano then all elements of $\gt^+_\bfw$ are indefinite.
\end{corollary}
\end{remark}

\begin{remark}
The bound $B$ described at the end of Section \ref{Sascone} is given by $l_1w_2$ when $N$ is Fano and has no Hamiltonian vector fields. Generally, if $N$ has Hamiltonian vector fields, the bound $l_1w_2$ is only a bound for the restriction of $B$ to the $\bfw$ subcone $\gt^+_\bfw\subset \gt^+$.
\end{remark}

\subsection{Examples} The simplest case is when the affine cone $(C(M)$ over $M$ is Gorenstein or $\bbq$-Gorenstein, or equivalently $c_1(\cald)_\bbq=c_1(\cald)\otimes\bbq=0$. By Theorem \ref{c10pos} the entire Sasaki cone is positive so there is no type changing. Moreover, it is well known \cite{MaSpYau06} that there is precisely one critical point of ${\bf H}_1(\xi)$ and it is a minimum with ${\bf H}_1(\xi_{min})>0$. A particularly well known case is that of $Y^{p,q}$ \cite{GMSW04a}. It was shown in Corollary 5.5 of \cite{BoPa10} that for $p$ fixed we have a $\phi(p)$-bouquet of Sasaki cones where $\phi(p)$ is the Euler phi function, and all the Sasakian structures of the bouquet are positive. Recall \cite{Boy10a} that bouquets occur by deforming the transverse complex structure.


The contact equivalence problem was studied in \cite{BoPa10} for the case of $S^3$ bundles over $S^2$. Even in this case the problem is far from completely solved. The notation there is quite different from that used in \cite{BoTo14a} and here where the underlying contact CR structure is labeled by the four integral parameters $(l_1,l_2;w_1,w_2)$. In \cite{BoPa10} the contact vector bundle is label by the three integers $k,j,l$. Moreover, the primitive generator of $H^2(M,\bbz)$ that we call $\grg$ here is $-\grg$ in \cite{BoPa10}. The relation between the labels are
$$l_1=\gcd(2k-j,j),~~l_2=l,~~l_1\bfw=(2k-j,j)$$
where $j=1,\ldots,k$.
This implies $j=l_1w_2=B$,  and $2k=l_1|\bfw|$. In \cite{BoPa10} the authors also defined a map $g$ which associates a positive integer $i$ to the set $\{j=1,\ldots,k\}$. Theorem 4.11 of \cite{BoPa10} shows that the elements of the level sets of $g$ are $T^2$ equivariantly equivalent but not $T^3$ equivariantly equivalent. This implies that the level set $g^{-1}(i)$ gives rise to a $k$-bouquet of Sasaki cones whose cardinality $k$ is the cardinality of $g^{-1}(i)$, see Corollary 5.3 of \cite{BoPa10}.  Note that the integer $i=g(j)=\gcd(l,2(k-j))$ is the integer $s$ for the almost regular Reeb vector field, i.e. $\bfv=(1,1)$.

\begin{remark}
The integer $k$ should be an invariant of the contact structure. If the problems involving transversality for the contact homology in this case can be resolved, Proposition 3.11 of \cite{BoPa10} would show that the integer $k$ is a contact invariant. Nevertheless, Theorem 4.11 of \cite{BoPa10} does give $k$ as a lower bound to the number of conjugacy classes of 3-dimensional tori in the contactomorphism group.
\end{remark}

We illustrate type changing with several examples. The first two are examples of bouquets.  For the case of $S^2\times S^3$ discussed in Examples \ref{posex} and \ref{posex2} below, see Corollary 5.3 of \cite{BoPa10}.
The third example does not explicitly involve bouquets. They are all joins of the form $S^{2p+1}\star_{l_1,l_2}S^3_\bfw$. The first two have $p=1$, whereas the third has $p>1$.
From item (1) in Proposition \ref{typech} we see that if $c_1(\cald)$ is non-negative, then all $\bfw$-cones  in a bouquet will be entirely positive. Hence, for contact structures with $c_1(\cald)\geq 0$ complete positivity of $\gt^+_{\bfl,\bfw}$ depends only on the bouquet and not on the Sasaki cones in the bouquet. In particular this holds for the cohomologically Einstein case $c_1(\cald)_\bbr= 0$ which depends only on the ray in $\gt^+$.

\begin{example}\label{posex}
Consider the bouquet on $S^2\times S^3$ described in \cite{BoTo14P}, namely 
\medskip

\centerline{$4$-bouquet on $S^2\times S^3$}
\begin{center}
\begin{tabular}{ | c |l| c | c |l c | c | c | c | }
\hline
$m$ & $l_1$ & \bfw & B \\ \hline
0 & 4 & (1,1) & 4 \\ \hline 
1 & 1 & (5,3) & 3 \\ \hline 
2 & 2 & (3,1) & 2 \\ \hline 
3 & 1 & (7,1) & 1 \\ \hline 
\end{tabular}
\end{center}
\medskip

\noindent Here we take $l_2=1$ with $m=\frac{1}{2}l_1(w_1-w_2)$, and we have $g^{-1}(1)=\{1,2,3,4\}$. For all members of the bouquet we have $c_1(\cald)=-6\grg$ where $\grg$ is a positive generator of $H^2(S^2\times S^3,\bbz)$. Note that we have $j=l_1w_2=B$ which give the four cases, and that $B+m=k=4$ which we believe to be a contact invariant. The positivity range $\gp^+_\bfw$ in the four cases are:

\begin{center}
\begin{tabular}{| l | l |}
\hline
$m$ & positivity range \\ \hline
$0$ & $\frac{1}{2}<\frac{v_1}{v_2}<2$ \\ \hline
$1$ & $1<   \frac{v_1}{v_2}<5$ \\ \hline
$2$ & $2<   \frac{v_1}{v_2}$ \\ \hline
$3$ & $5<   \frac{v_1}{v_2}$ \\ \hline
\end{tabular}
\end{center}

\noindent We note that the regular ray $\bfv=(1,1)$ occurring  in each element of the bouquet corresponds to an $S^1$-bundle of $S^2\times S^3$ over the even Hirzebruch surface $\calh_{2m}$ for $m=0,1,2,3$. Of course, only when $m=0$ is this ray positive. It follows from Proposition \ref{c1wc1v} that $H_1(\xi_{min})>0$ for all Sasaki cones of the bouquet. Moreover, from Section 5.3 of \cite{BHLT15} we have critical points of ${\bf H}_1$ that lie in the $\bfw$ subcone and which all have constant scalar curvature. They are all minimum restricted to their $\bfw$-Sasaki cone, but we do not know whether they are a global minimum. For $m=0,1$ the minima are positive Sasakian structures; whereas, for $m=2,3$ they are indefinite Sasakian structures.
\end{example}

Next we give an example where for one element of a bouquet the entire $\bfw$-cone is positive, but in the other it is not.

\begin{example}\label{posex2}
Again our manifold is $S^2\times S^3$, but now $l_2=3$ with the contact structure satisfying $c_1(\cald)=-2\grg$. We also need to satisfy the smoothness condition \cite{BoTo14a} $\gcd(3,l_1w_1w_2)=1$. In this case we have a regular two-bouquet:
\medskip

\centerline{$2$-bouquet on $S^2\times S^3$}
\begin{center}
\begin{tabular}{ | c |l| c | c |l c | c | c | c | }
\hline
$m$ & $l_1$ & \bfw & B \\ \hline
0 & 4 & (1,1) & 4 \\ \hline
3 & 1 & (7,1) & 1 \\ \hline
\end{tabular}
\end{center}
\medskip

\noindent For $m=0$ we have $l_2\cali_N=6>4=l_1w_1$, so the entire $\bfw$-cone is positive. However, for $m=3$ we have $l_2\cali_N=6<7=l_1w_1$, so this has positivity range $1<\frac{v_1}{v_2}$.
\end{example}

\begin{example}\label{pex}
Consider what happens when we take the $S^3_\bfw$ join with an odd dimensional sphere $S^{2p+1}$ of arbitrary dimension, namely, $S^{2p+1}\star_{l_1,l_2}S^3_\bfw$.
When $p>1$ the integral cohomology ring is
\begin{equation}\label{cohring}
\bbz[x,y]/(w_1w_2l_1^2x^2,x^{p+1},x^2y,y^2)
\end{equation} 
where $x$ is a 2-dimensional class and $y$ is a $2p+1$ dimensional class.
So the topology of the manifold $M_{\bfl,\bfw}$ depends on the product $w_1w_2l_1^2$. Moreover, it follows from Sullivan's rational homotopy theory \cite{Sul77,BoTo14b} that with $w_1w_2l_1^2$ fixed there are a finite number of diffeomorphism types as $l_2$ varies through positive integers that are relatively prime to $l_1w_1w_2$. But generally the contact structure can also vary with $l_2$ as we have 
$$c_1(\cald_{l_1,l_2,\bfw})= (l_2(p+1)-l_1|\bfw|)\grg.$$
The case $c_1(\cald_{l_1,l_2,\bfw})= 0$ can be realized by taking $l_1=p+1$ and $l_2=|\bfw|=w_1+w_2$ and was studied in Section 3.1.1 of \cite{BoTo14NY}. Here the entire $\bfw$ Sasaki cone $\gt^+$ is positive by Theorem \ref{c10pos}.

In this case we have little control over whether or not bouquets exist. Nevertheless, the type changing described by Proposition \ref{typech} still occurs. Let us describe this with a simple example. Put $p=2$ and $w_1w_2l_1^2=12$. This gives manifolds $M_{1,l_2,12,1}, M_{1,l_2,4,3},M_{2,l_2,3,1}$ with isomorphic cohomology rings $\bbz[x,y]/(12x^2,x^{p+1},x^2y,y^2)$ with $\gcd(l_2,6)=1$. The corresponding contact structures satisfy 
\begin{equation}\label{c1ex}
c_1(\cald_{1,l_2,12,1})=(3l_2-13)\grg,~ c_1(\cald_{1,l_2,4,3})=(3l_2-7)\grg,~c_1(\cald_{2,l_2,3,1})=(3l_2-8)\grg.
\end{equation}
So the contact structure varies with $l_2$. Notice that since $l_2$ must be odd the second Stiefel-Whitney class $w_2$ vanishes for $M_{1,l_2,12,1}$ and $M_{1,l_2,4,3}$, and does not vanish for $M_{2,l_2,3,1}$ which implies that $M_{2,l_2,3,1}$ is not homotopy equivalent to $M_{1,l_2,12,1}$ or $M_{1,l_2,4,3}$.
We do not know at this stage whether the manifolds $M_{1,l_2,12,1}$ and $M_{1,l_2,4,3}$ are diffeomorphic, homeomorphic, or even homotopy equivalent.
Note, however, that we have $c_1(\cald_{1,l'_2,12,1})=(3l'_2-13)\grg= c_1(\cald_{1,l_2,4,3})=(3l_2-7)\grg$ if and only if $l'_2=l_2+2$. 

One can easily determine the ranges of positivity for these contact manifolds. For example for $M_{1,l_2,12,1}$ we have the positivity range as follows $\gp^+_\bfw=\gt^+_\bfw$ if $l_2\geq 4$, whereas, for $l_2=1,2,3$, the positivity subcone $\gp^+_\bfw$ is given by $v_1/v_2>9,6,3$, respectively. Similarly, for $M_{1,l_2,4,3}$ we have $\gp^+_\bfw=\gt^+_\bfw$ for $l_2\geq 2$ and $\gp^+_\bfw$ given by $v_1/v_2>1/3$ for $l_2=1$. Notice that $M_{1,3,12,1}$ and $M_{1,1,4,3}$ could belong to a bouquet, but we have not proven this.

Other similar examples can easily be worked out. 
\end{example}

\appendix\section{Admissible metrics with positive Ricci curvature}\label{positive explicit}

The purpose of this appendix is to demonstrate how one can easily produce explicit admissible examples of Sasaki metrics with positive Ricci curvature on certain joins.
In order to do so, we need to first give a quick introduction to admissible K\"ahler metrics (Section \ref{prelimapp}). Then we will use this to produce explicit orbifold admissible K\"ahler 
metrics with positive Ricci curvature (Section \ref{explicitKPosRic}). Finally, we mention how this construction can be lifted to the Sasaki level (Section \ref{sasapp}). 

Smooth admissible K\"ahler metrics were defined in \cite{ACGT08}. In its full generality it combines the formalism of various constructions \cite{Gua95, Hwa94, HwaSi02, KoSa86,LeB91b,PePo91,To-Fr98,Sa86} that in turn generalized a well-known construction, used by Calabi~\cite{Cal82} to construct extremal K\"ahler metrics on the Hirzebruch complex surfaces. Here we will not need the full generality of the admissible set-up (we have only one piece in the base and no "blow-downs"), but we will extend to a mild orbifold case.

\subsection{Preliminaries}\label{prelimapp}
Let $\omega_N$ be a primitive integral  K\"ahler form of a 
CSC K\"ahler metric on $(N,J)$,  ${\BOne} \rightarrow N$ be the 
trivial complex line bundle, $n \in {\bbz}\setminus \{ 0 \}$, and let $ L_n \rightarrow N$ 
be a holomorphic line bundle with $c_1( L_n) = [n\,\omega_N]$.
Consider the total space of a projective bundle $S_n={\mathbb P}({\BOne} 
\oplus L_n)
\rightarrow N$.
Then $S_n$ is called {\bf admissible}, or an {\bf admissible 
manifold} \cite{ACGT08}.
Now $D_{1} = [\BOne \oplus 0]$ and $D_{2}=[0\oplus L_n]$ are the so-called ``zero'' and ``infinity'' 
sections of $S_n \rightarrow N$.

Let $r$ be a real number 
such that $0< |r|< 1$, and such that $r\, n>0$.
A K\"ahler class on $S_n$,
$\Omega$, is {\bf admissible} if (up to scale)
$$\Omega =  \frac{2\pi n[\omega_{N}]}{r} + 2\pi PD(D_1 + D_2).$$
In general, the {\bf admissible cone} is a sub-cone of the K\"ahler cone.

In each admissible 
class we can now construct {\bf explicit} K\"ahler metrics $g$ (called {\bf admissible K\"ahler metrics})
\cite{ACGT08}.
We can generalize this construction to the log pair $(S_n,\Delta)$,
where $\Delta$ denotes the branch divisor
$\grD= (1-1/m_1)D_1+(1-1/m_2)D_2.$
If $m=\gcd(m_1,m_2)$, then $(S_n,\Delta)$ is a fiber bundle over $N$ with fiber $\bbc\bbp^1[m_1/m,m_2/m]/\bbz_m$.
The admissible metric is smooth on $S_n\setminus (D_1 \cup D_2)$ and has orbifold singularities along $D_1$ and $D_2$.

More specifically, on $S_n\setminus (D_1 \cup D_2)$, an admissible K\"ahler metric with corresponding K\"ahler form is
given up to scale by 
\begin{equation}\label{g}
g=\frac{1+r\gz}{r}g_{N_n}+\frac {d\gz^2}
{\Theta (\gz)}+\Theta (\gz)\theta^2,\quad
\omega = \frac{1+r\gz}{r}\omega_{N_n} + d\gz \wedge
\theta,
\end{equation}
where $(g_{N_n},\omega_{N_n})= (2\pi ng_N,2\pi n \omega_N)$, $\gz$ can be interpreted as a moment map of the natural $S^1$-action, $\theta$ is a connection $1$-form,
and $\Theta$ is a positive smooth function on $(-1,1)$
(for precise definitions, please consult \cite{ACGT08}). Now $\Theta (\gz)$ must satisfy certain endpoint conditions in order for
$g$ to extend as an orbifold metric on $(S_n,\Delta)$:
$$
\begin{array}{l}
\Theta(\pm 1) = 0,\\
\\
\Theta'(- 1) = 2/m_2 \quad \quad \Theta'( 1) =-2/m_1.
\end{array}
$$
Define a function $F(\gz)$ by the formula $\Theta(\gz)=F(\gz)/\gp(\gz)$, where
$\gp(\gz) =(1 + r \gz)^{d_{N}}$.
Since $\gp(\gz)$ is positive for $-1\leq \gz \leq1$, the conditions
on $F(\gz)$ are:
\begin{equation}
\label{positivityF}
\begin{array}{l}
F(\gz) > 0, \quad -1 < \gz <1,\\
\\
F(\pm 1) = 0,\\
\\
F'(- 1) = 2\gp(-1)/m_2 \quad \quad F'( 1) =-2\gp(1)/m_1.
\end{array}
\end{equation}

From \cite{ApCaGa06} we have that the Ricci form of an admissible metric  given by \eqref{g} equals
\begin{equation}\label{rho}
\rho =
\rho_{N} - \frac{1}{2} d d^c \log F = \rho_N - 
\frac{1}{2}\frac{F'(\gz)}{\gp(\gz)}  \omega_{N_n}
-\frac{1}{2}\Bigl(\frac{F'(\gz)}{\gp(\gz)}\Bigr)' d\gz \wedge \theta.
\end{equation}

Note that $r$, $n$ and $\omega_{N_n}$ all have the same sign. Let us assume $N$ is KE so that 
$\rho_N=s_{N_n}\omega_{N_n}$ with
$s_{N_n} =\cali_N/n$.
Then
\begin{equation}\label{rho2}
\rho =
 \left(\cali_N/n - 
\frac{1}{2}\frac{F'(\gz)}{\gp(\gz)} \right) \omega_{N_n}
-\frac{1}{2}\Bigl(\frac{F'(\gz)}{\gp(\gz)}\Bigr)' d\gz \wedge \theta.
\end{equation}
Since $\omega$ and $\frac{1+r\gz}{r}\pi^*\omega_{N_n}$ are both globally defined on $S_n$, then so is $d\gz \wedge \theta$ (even though $\theta$ is not). Thus $\rho$ is a globally defined $(1,1)$-form.

On $S_n\setminus (D_1 \cup D_2)$, the $(1,1)$-forms  $\omega_{N_n}$ and $dz\wedge \theta$ are orthonormal with respect to the orthogonal splitting $TS_n = H \oplus V$ defined by the global K\"ahler metric $g$, where $V=Ker(\pi_*)$ is the vertical space.
Specifically, $dz\wedge \theta \in \wedge^2 V^*$ and $\pi^*\omega_{N_n} \in \wedge^2 H^*$. This must then be true globally by continuity. 
Therefore\footnote{The authors would like to thank Vestislav Apostolov for his kind advice on this argument.}, $\rho$ is positive if and only if for all $z\in [-1,1]$
\begin{equation}\label{cond1}
\left(\cali_N/n -\frac{1}{2}\frac{F'(\gz)}{\gp(\gz)}\right)\cdot n >0
\end{equation}
and
  
\begin{equation}\label{cond2}
\Bigl(\frac{F'(\gz)}{\gp(\gz)}\Bigr)'<0.
\end{equation}
  
In particular, we must have that \eqref{cond1} holds at $z=\pm 1$, so we must have that
$$\left(\cali_N/n -\frac{1}{m_2}\right)\cdot n >0\quad\text{and}\quad \left(\cali_N/n + \frac{1}{m_1}\right)\cdot n >0.$$ 
We see right away, as is already obvious, that this will not happen if $\cali_N \leq 0$. Thus moving forward we assume that $\cali_N >0$. 
Now, if 
$$\left(\cali_N/n -\frac{1}{m_2}\right)\cdot n >0\quad\text{and}\quad \left(\cali_N/n + \frac{1}{m_1}\right)\cdot n >0,$$ 
then  \eqref{cond1} holds at $z=\pm 1$ and then if \eqref{cond2} is also satisfied one can check that \eqref{cond1} must hold for all
$z\in[-1,1]$.
All in all, $\rho$ is positive if and only if
$$\boxed{\left(\cali_N/n -\frac{1}{m_2}\right)\cdot n >0,\quad
\left(\cali_N/n + \frac{1}{m_1}\right)\cdot n >0, \quad
\Bigl(\frac{F'(\gz)}{\gp(\gz)}\Bigr)'<0}$$
Notice that if $n>0$, then the middle condition is automatic while if $n<0$, the first equation is automatic. 
As we saw in the proof of Proposition \ref{typech},  the first two conditions in the box above are equivalent to $c_1^{orb}(S_n,\Delta)$ being positive.

\subsection{Explicit Positive Ricci Curvature Admissible K\"ahler Metrics}\label{explicitKPosRic}

We now assume that $c_1^{orb}(S_n,\Delta)$ is positive. The existence of some positive Ricci curvature K\"ahler metric is well established by Theorem 7.5.19 in \cite{BG05}, but to see an explicit admissible example we need to produce a function $F(\gz)$ satisfying \eqref{positivityF} and \eqref{cond2}. Obviously admissible 
K\"ahler-Einstein metrics - should they exist - would work, so we look to K\"ahler-Einstein constructions \cite{KoSa86} for inspiration (See also Proposition 5.4 in \cite{BoTo14a}). What follows may be viewed as a generalization (different from the usual generalizations) of admissible K\"ahler-Einstein. 

Consider the following function
$$g(t,k) = \left\{ \begin{array}{cc}
2\frac{(\frac{1}{m_1} + \frac{1}{m_2})e^{-kt} - (\frac{e^k}{m_1} + \frac{e^{-k}}{m_2})}{e^k - e^{-k}}& \text{if}\,\,k\neq 0\\
\\
\frac{1-t}{m_2} - \frac{1+t}{m_1} & \text{if}\,\, k=0.
\end{array}\right.$$
One may check that this is continuously differentiable for all $(t,k) \in \bbr^2$.  Further, for any value of $k\in \bbr$, $g(t,k)$ is clearly a smooth function of $t\in \bbr$.

Now observe that 
\begin{itemize}
\item $$\frac{\partial g(t,k)}{\partial t} = \left\{ \begin{array}{cc}
-2k\frac{(\frac{1}{m_1} + \frac{1}{m_2})e^{-kt} }{e^k - e^{-k}}& \text{if}\,\,k\neq 0\\
\\
-(\frac{1}{m_1} + \frac{1}{m_2}) & \text{if}\,\, k=0.
\end{array}\right.$$
is always negative. 
\item $$g(-1,k) = \frac{2}{m_2}\quad \text{and} \quad g(1,k) = \frac{-2}{m_1}.$$

\item For $-1<t<1$, 
$$\lim_{k\rightarrow +\infty} g(t,k) = \frac{-2}{m_1}, \quad \lim_{k\rightarrow -\infty} g(t,k) = \frac{2}{m_2},$$
and
$$\frac{\partial g(t,k)}{\partial k} <0.$$
\end{itemize}

\begin{figure}[h!]
  \caption{$g(t,k)$ for $m_1=m_2=1$ and some $k$ values between -100 and 100}
  \centering
\includegraphics[height=1.2in]{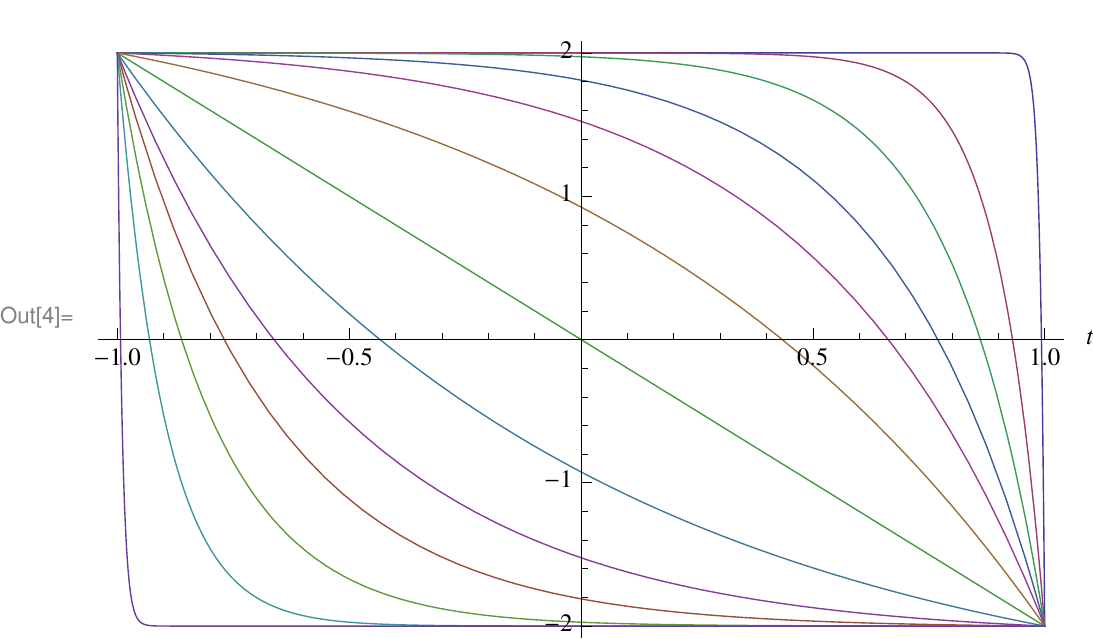}
\end{figure}

\begin{lemma}\label{tool}
There exists a unique $k_{m_1,m_2}\in \bbr$ such that
$$\int_{-1}^{1} g(t,k_{m_1,m_2}) \gp(t)\, dt = 0$$
\end{lemma}
 
 \begin{proof}
 Recall that $\gp(\gz) >0$ for $-1\leq \gz \leq 1$.
 Now we define the function
 $$f(k) = \int_{-1}^{1} g(t,k) \gp(t)\, dt.
 $$
 The facts listed above gives us that $f(k)$ is continuous for all $k\in \bbr$ and
 $$\lim_{k\rightarrow +\infty} f(k)  = \frac{-2}{m_1} \int_{-1}^{1} \gp(t)\, dt <0$$
while
$$\lim_{k\rightarrow -\infty} f(k)  = \frac{2}{m_2} \int_{-1}^{1} \gp(t)\, dt > 0.$$
Since moreover
$$f'(k) =  \int_{-1}^{1} \frac{\partial g(t,k)}{\partial k}\gp(t)\, dt <0,$$
we conclude that $f(k)$ vanishes for precisely one $k$-value. This proves the lemma.
 \end{proof}

Note that $f(0)=0$ iff
 $$\int_{-1}^1 (\frac{1-t}{m_2} - \frac{1+t}{m_1})\gp(t)\,dt =0.$$
In the event where we also know that
$$2r\cali_N/n = (1+r)/m_2 + (1-r)/m_1,$$
this corresponds to the existence of admissible K\"ahler-Einstein examples given by Proposition 5.4 in 
\cite{BoTo14a}.

In general, we may now define an admissible metric by choosing $F(\gz)$ as follows:
$$F(\gz) = \int_{-1}^\gz g(t,k_{m_1,m_2}) \gp(t)\,dt.$$

Note that 
$$F(\pm 1) =0,$$
$$F'(\gz) = g(\gz,k_{m_1,m_2}){\gp(\gz)},$$
and hence 
$$F'(- 1) = 2\gp(-1)/m_2 \quad \quad F'( 1) =-2\gp(1)/m_1.$$
Moreover, since $\gp(t)>0$ and $g(t,k_{m_1,m_2})$ is a monotone function over $[-1,1]$, which is positive at $\gz=-1$, negative at $\gz = 1$, and satisfies that $\int_{-1}^{1} g(t,k_{m_1,m_2}) \gp(t)\, dt = 0$, we realize that
$F(\gz)>0$ for $-1<\gz <1$. Thus $F(\gz)$ satisfies all the conditions in \eqref{positivityF}.
Finally 
$$\Bigl(\frac{F'(\gz)}{\gp(\gz)}\Bigr)' = \frac{\partial g(\gz,k_{m_1,m_2})}{\partial \gz} <0,$$
and so we have our explicit admissible example with positive Ricci curvature.

\subsection{Explicit Sasaki metrics with positive Ricci curvature}\label{sasapp}
Now consider the $S^3_\bfw$-join $M_{\bfl,\bfw}$ and assume $N$ is a positive KE manifold.
Let $\bfv=(v_1,v_2)\neq \bfw$ be a weight vector with relatively prime integer components and let $\xi_\bfv$ be the corresponding Reeb vector field in the $\bfw$-Sasaki cone $\gt^+_\bfw$. 
Then from \cite{BoTo14a}, we know that the quotient of $M_{l_1,l_2,\bfw}$ by the flow of the Reeb vector field $\xi_\bfv$ is 
a certain log pair $(S_n,\grD)$ (where $n\in \bbz\setminus\{0\}$ and $m_i = m v_i$) with admissible K\"ahler class. 
Assume that $(v_1,v_2)$ is such that $c_1^{orb}(S_n,\Delta)$ is positive.
Then the explicit orbifold K\"ahler metrics with positive Ricci curvature produced in Section \ref{explicitKPosRic}, gives us (after an appropriate homothety transformation) explicit quasi-regular admissible Sasaki metrics with positive Ricci curvature in the ray
of $\xi_\bfv$.  Further, since for the explicit $\Theta(\gz) = F(\gz)\gp(\gz)$, produced in Section \ref{explicitKPosRic}, we have that $m\Theta(\gz)$ is independent of $m$ and varies smoothly with $(v_1,v_2)$, we can lift the admissible construction to the Sasakian level (as in Section 6 of  \cite{BoTo14a}) and get explicit  irregular examples as well.

\def\cprime{$'$} \def\cprime{$'$} \def\cprime{$'$} \def\cprime{$'$}
  \def\cprime{$'$} \def\cprime{$'$} \def\cprime{$'$} \def\cprime{$'$}
  \def\cdprime{$''$} \def\cprime{$'$} \def\cprime{$'$} \def\cprime{$'$}
  \def\cprime{$'$}
\providecommand{\bysame}{\leavevmode\hbox to3em{\hrulefill}\thinspace}
\providecommand{\MR}{\relax\ifhmode\unskip\space\fi MR }
\providecommand{\MRhref}[2]{%
  \href{http://www.ams.org/mathscinet-getitem?mr=#1}{#2}
}
\providecommand{\href}[2]{#2}

\end{document}